\documentclass{amsart}
\usepackage{amsmath,amstext,amssymb,amsfonts,amscd,graphicx,cases}
\usepackage{mathrsfs,accents,mathtools}
\newtheorem{thm}{Theorem}[section]
\newtheorem{cor}[thm]{Corollary}
\newtheorem{lem}[thm]{Lemma}
\newtheorem{prop}[thm]{Proposition}
\theoremstyle{definition}
\newtheorem{defn}[thm]{Definition}
\theoremstyle{remark}

\numberwithin{equation}{section}
\title[Well-posedness and scattering for the Boltzmann equations]
{Well-posedness and scattering for the Boltzmann equations: Soft potential with cut-off}

\author{Lingbing He}
\address{Department of Mathematical Sciences, TsingHua University, Beijing, 100084, P.R.China}
\email{lbhe@math.tsinghua.edu.cn}

\author{Jin-Cheng Jiang}
\address{Department of Mathematics, National Tsing Hua University, Hsinchu, Taiwan 30013, R.O.C}
\email{jcjiang@math.nthu.edu.tw}

\begin{document}

\begin{abstract}
We prove the global existence of the  unique mild solution for the Cauchy problem 
of the cut-off Boltzmann equation for soft potential model $\gamma=2-N$ 
with initial data  small in $L^N_{x,v}$ where $N=2,3$ is the dimension. 
The proof relies on the existing inhomogeneous 
Strichartz estimates for the kinetic equation by Ovcharov~\cite{Ovc11} and 
convolution-like estimates for the gain term of the Boltzmann collision operator by 
Alonso, Carneiro and Gamba~\cite{ACG10}. The global dynamics of the solution is also 
characterized by showing that the small global solution scatters 
with respect to the kinetic transport operator in $L^N_{x,v}$. 
Also the connection between function spaces and  cut-off 
soft potential model $-N<\gamma<2-N$ is
characterized  in the local well-posedness result for the Cauchy problem
 with large initial data.

\end{abstract}

\maketitle

\section{Introduction and Results}

With the first appearance of 
Strichartz estimates for the kinetic equation in the note of Castella and 
Perthame~\cite{CP96}, the Strichartz estimates have been  applied 
to proving the existence of global weak solution with small initial data assumption for
the kinetic equation,
Bournaveas et al.~\cite{BCGP08} for a nonlinear kinetic system modeling chemotaxis
and Ars\'{e}nio~\cite{Ars11} for the cut-off Boltzmann equation. We note that the result of 
 Ars\'{e}nio holds only for non-conventional collision kernel 
 whose kinetic part is $L^p$ integrable for some $p$ depending on 
dimension and the weak solution  is not unique.

We accomplish this approach to some extend for the case of the Boltzmann equation
by proving the global existence of the  unique mild solution for the Cauchy problem 
of  the cut-off Boltzmann equation for soft potential model $\gamma=2-N$ with initial data  small in $L^N_{x,v}$ where $N=2,3$ is the dimension. The proof relies on the existing inhomogeneous 
Strichartz estimates for the kinetic equation by Ovcharov~\cite{Ovc11} and 
convolution-like estimates for the gain term of the Boltzmann collision operator by 
Alonso, Carneiro and Gamba~\cite{ACG10}. 
The global dynamics of the solution is also 
characterized by showing that small global solution scatters 
with respect to the kinetic transport operator in $L^N_{x,v}$.
Also the connection between function spaces and  cut-off 
soft potential model $-N<\gamma<2-N$ is
characterized  in the local well-posedness result for the Cauchy problem
 with large initial data.

To state the results precisely, we begin with the introduction of  the necessary
notations.  We consider the Cauchy problem for the Boltzmann equation 
\begin{equation}\label{E:Cauchy}
\left\{
\begin{aligned}
&{\partial_t f}+v\cdot\nabla_x f =Q(f,f)\\
&f(0,x,v)=f_0(x,v)
\end{aligned}
\right.
\end{equation}
in $(0,\infty)\times\mathbb{R}^N\times\mathbb{R}^N, N=2,3,$ 
where the collision operator  
\[
Q(f,f)(v)=\int_{\mathbb{R}^N}\int_{\omega\in S^{N-1}}(f'f'_{*}-ff_{*})B(v-v_{*},\omega)d\Omega(\omega)dv_{*},
\]
and $d\Omega(\omega)$ is the solid element in the direction of unit vector $\omega$.
Here we have used the abbreviations $f'=f(x,v',t)\;,\;f'_{*}=f(x,v'_{*},t)\;,\;f_{*}=f(x,v_{*},t)$, 
where the relation between the pre-collisional velocities of particles and
after collision is given by
\[
v'=v-[\omega\cdot(v-v_{*})]\omega\;,\;v'_{*}=v_{*}+[\omega\cdot(v-v_{*})]\omega\;,\;\omega\in S^{N-1}.
\]
The  cut-off soft potential collision kernel takes the from 
\begin{equation}\label{D:kernel}
B(v-v_*,\omega)=|v-v_*|^{\gamma}b(\cos\theta),\;0\leq\theta\leq \pi/2
\end{equation}
where 
\[
-N<\gamma<0\;,\;\cos\theta=\frac{(v-v_*)\cdot\omega}{|v-v_*|}
\]
and the angular function $b$ satisfies the Grad's cut-off assumption 
\begin{equation}\label{D:Grad}
\int_{S^{N-1}} b(\cos\theta)d\Omega(\omega)<\infty.
\end{equation}
When $\gamma=0$, ~\eqref{D:kernel} is called the Maxwell molecules. 
For our purpose, we introduce the mixed Lebesgue norm 
\[
\|f(t,x,v)\|_{L^q_tL^r_xL^p_v}
\] where the notation  $L^q_tL^r_xL^p_v$ stands for the space $L^q(\mathbb{R};L^r
(\mathbb{R}^N;L^p(\mathbb{R}^N)))$ and it is understood that we are using 
$L^q_t(\mathbb{R})=L^q_t([0,\infty))$ for the well-posedness problem which can 
be done by imposing support restriction to the inhomogeneous  Strichartz estimates. 
We use $L^a_{x,v}$ to denote $L_x^a(\mathbb{R}^N;L_v^a(\mathbb{R}^N))$.

We also need to define the meaning of the solution scatters with
respect to kinetic transport operator in our result. It seems strange to 
mention the notion of the scattering of the solution of the Boltzmann equation since it involves 
Boltzmann's H-theorem (see for example~\cite{BGGL16} for more discussion). 
From the mathematical point of view, the solution scatters implies  that the hyperbolic part of 
the equation dominates the solution all the time and the definition of scattering thus 
help us to understand the  large time behavior of solution, see also remark 4 after Theorem~\ref{T:1}
and Corollary~\ref{Cor1}.       
Here we  say that a global solution $f\in C([0,\infty),L^a_{x,v})$ scatters in $L^a_{x,v}$ as 
$t\rightarrow\infty$ if there exits $f_+\in L^a_{x,v}$ such that 
\begin{equation}
\|f(t)-U(t)f_+\|_{L^a_{x,v}}\rightarrow 0
\end{equation}
where $U(t)f(x,v)=f(x-vt,v)$ is the solution map of the kinetic transport equation 
\[
\partial_t f+v\cdot\nabla_x f=0. 
\]
We note that the operator $U(t)$ is time reversible and thus the scattering problem is still well-defined if we consider 
$t\rightarrow -\infty$ or $t$ goes from $-\infty$ to $\infty$. Since the results for these scattering problems are similar, we only 
present the case $t\rightarrow\infty$.    

The main result of this paper is the following.
\begin{thm}\label{T:1} 
Let $N=2$ or $3$ and $B$ defined in~\eqref{D:kernel} satisfies~\eqref{D:Grad} and $\gamma=2-N$. 
The Cauchy problem~\eqref{E:Cauchy} is 
globally wellposed in $L^N_{x,v}$ when the initial data is small enough. 
More specially, there exists $R>0$ small enough  such that for all $f_0$ in the ball $B_{R}=\{f_0\in L^N_{x,v} (\mathbb{R}^N\times
\mathbb{R}^N):\|f_0\|_{L^N_{x,v}}<R\}$ there exists a globally unique  mild solution 
\[
f\in C([0,\infty),L^N_{x,v})\cap
L^q([0,\infty],L^r_xL^p_v) 
\]
where the triple $(q,r,p)$ lies in the set
\begin{equation}
\{ (q,r,p) | \;\frac{1}{q}=\frac{N}{p}-1\;,\;\frac{1}{r}=\frac{2}{N}-\frac{1}{p}\;,\;
\frac{1}{N}<\frac{1}{p}<\frac{N+1}{N^2}\}. 
\end{equation} 
The solution map $f_0\in B_{R}\subset L^N_{x,v}\rightarrow f\in L^q_tL^r_xL^p_v$ 
is Lipschitz continuous and the solution $f$ scatters with respect to the kinetic 
transport operator in $L^N_{x,v}$.
\end{thm}
Some comments   about this result are given in the following.

1. The related earlier work using the the iterative scheme proposed by Kaniel and 
Shinbrot~\cite{KS78,IS84,BT85,Pol88,AG09,AG11} or fixed point argument~\cite{Ham85} all require
pointwise upper bound. One of advantages of the current approach is that it  
requires only the initial date is small in $L^N_{x,v}$.

2. Ars\'{e}nio~\cite{Ars11} noted that $L^3_x$ appeared in the well-posed result of the Boltzmann equation matches the critical space 
of Navier-Stokes equation~\cite{Kato84}. 
It should be interesting to see how are they related. 
On the other hand, the generalized homogeneous Strichartz estimate~\cite{Ovc11}
reads 
\[
\|U(t)f_0\|_{L^q_tL^r_xL^p_v}\leq C \|f_0\|_{L^b_xL^c_v},
\] 
where 
\[
\begin{split}
&\frac{1}{q}+\frac{N}{r}=\frac{N}{b},\;HM(p,r)=HM(b,c)\;
{\stackrel{\mathclap{\normalfont\mbox{def}}}{=}}\;a\\
& p<b\leq a\leq c<r,
\end{split}
\]
with the definition of $HM(p,r)$  given in Definition~\ref{D:admissible} below. 
This allows us to choose the initial data  in $L^b_xL^c_v$ space 
where $b$ is less than $N$ by paying the price of rising $c$. It is not clear 
if this flexibility of choosing initial data in such spaces really 
reflects the difference between kinetic equation and hydrodynamic equation. 
Hence we retain the statement of 
the initial data in the current format.

3. Since the property of loss term is not fully utilized in our analysis,  
 the exponent $\gamma=2-N$ is the number where the dispersive effect from 
 the kinetic transport part of the equation dominates the self-produced part
 from the collision operator when the small initial data is given.  
 We expect that this mechanism should work for a more wide range of soft 
 potential kernel if the loss term is properly used.
  One the other hand, the $v$ variable estimates 
    for the gain term of the Boltzmann collision operator with hard potential ,$0<\gamma\leq 1$, (see~\cite{JC12} and reference therein)    
 \[
 \|Q^+(f,f)(v)\|_{H^{\gamma-\varepsilon}}\leq C\|f(v)\|_{L^1_{\gamma}}\|f\|_{L^2_\gamma} 
 \]  
 suggests that the study of weighted Strichartz estimates is needed for applying such an approach to 
 hard potential or hard sphere case.

4. The  uniqueness of the small global solution implies that given a 
small enough scattering state $f_+\in L^N_{x,v}$, there exists a unique small 
enough initial data $f_0\in L^N_{x,v}$ whose corresponding global well-posed solution 
scatters to $U(t)f_+$ as $t\rightarrow\infty$ (see the proof of the Theorem~\ref{T:1}).
In fact we can define the wave operator 
\begin{equation}\label{D:wave-operator}
\Omega_+:B_{\overline{R}}\subset L^N_{x,v}\rightarrow B_{R}\subset L^N_{x,v}
\end{equation}
by $\Omega_+f_+=f_0$ and have the following result.
\begin{cor}\label{Cor1}
There exist $R,\overline{R}$ small enough such that the wave operator~\eqref{D:wave-operator} is one-to 
one and onto. 
\end{cor}

The next result is the local well-posedness for the large data Cauchy problem.  
\begin{thm}\label{T:2}
Let $N=2$ or $3$ and $B$ defined in~\eqref{D:kernel} satisfies~\eqref{D:Grad}  and 
$-N<\gamma<2-N$.  The Cauchy problem~\eqref{E:Cauchy} is locally wellposed
in $L^a_{x,v},\;  a={2N}/({\gamma+N})$. 
More specially, for any $R>0$ there exists 
a $T=T(r,p,R)$ such that for all $f_0$ in the ball $B_{R}=\{f_0\in L^a_{x,v}
(\mathbb{R}^N\times\mathbb{R}^N):\|f_0\|_{L^a_{x,v}}<R\}$ there exist 
$T\in(0,\infty]$ and a unique mild solution 
\[
f\in C([0,T),L^a_{x,v})\cap
L^q([0,T],L^r_xL^p_v) 
\]
where the triple $(q,r,p)$ lies in the set
\begin{equation}
\begin{split}
\{ (\frac{1}{q},\frac{1}{r},\frac{1}{p})  | \frac{1}{q}&=\frac{(2\alpha-1)(\gamma+N)}{2} \\
\frac{1}{r}&=\frac{(1-\alpha)(\gamma+N)}{N}\;,\;\frac{1}{p}
=\frac{\alpha(\gamma+N)}{N}\;,{\rm with}\;\frac{1}{2}<\alpha<\frac{N+1}{2N}\}.
\end{split}
\end{equation} 
The solution map $f_0\in B_{R}\subset L^N_{x,v}\rightarrow f\in L^q([0,T];L^r_xL^p_v)$ 
is Lipschitz continuous.
\end{thm}

Some comments about this result are given in the following. 

1. Heuristically,  the solutions  for large initial data exist 
during short time when the self-reproduced effect is not strong enough
and the hyperbolic part of equation dominates the solution.
This holds especially for soft collision where the non-local property of the 
collision operator is weaker. 
The result here indicates that the function spaces for the solutions
depending on the exponent of kinetic part of the collision kernel at least for very beginning of evolution.  
It seems that this intuition and  the fact 
$a\rightarrow\infty$ when $\gamma\rightarrow -3$ suggests 
a local well-posed result in $L^{\infty}_{x,v}$ for Landau equation.

2. Since the initial data lie in $L^a_{x,v}$ space, it is suitable to discuss the propagation of singularity of solution in this 
setting though we are not pursuing it here.

\section{Proof of the Theorems}

In order to prove Theorem~\ref{T:1} and~\ref{T:2}, we 
need to introduce the  Strichartz estimates for the kinetic transport equation
\begin{equation}\label{E:KT}
\left\{
\begin{aligned}
&\partial_t u(t,x,v)+v\cdot\nabla_x u(t,x,v) =F(t,x,v),\;\;(t,x,v)\in (0,\infty)\times\mathbb{R}^N\times\mathbb{R}^N,\\
& u(0,x,v)=u_0(x,v).
\end{aligned}
\right.
\end{equation}
To state the Strichartz estimates for the kinetic transport equation~\eqref{E:KT}, 
we need the following definition.
\begin{defn}\label{D:admissible}
We say that the exponent triplet $(q,r,p)$, for $1\leq p,q,r\leq\infty$ is KT-admissible if 
\begin{equation}
\frac{1}{q}=\frac{N}{2}{\Big ( \frac{1}{p}-\frac{1}{r} }{\Big )}
\end{equation}
\[
1\leq a\leq\infty,\;\;p*(a)\leq p\leq a, \;\;a\leq r\leq r*(a)
\]
except in the case $N=1,\;(q,r,p)=(a,\infty,a/2)$. Here by $a=$HM$(p,r)$ we have denoted the harmonic 
means of the exponents $r$ and $p$, i.e.,
\[
\frac{1}{a}=\frac{1}{2}{\Big ( \frac{1}{p}+\frac{1}{r} }{\Big )}
\]
Furthermore, the exact lower bound $p*$ to $p$ and the exact upper bound $r*$ to $r$ are 
\[
\left\{
\begin{array}{lll}
p*(a)=\frac{Na}{N+1}, & r*(a)=\frac{Na}{N-1} & {\rm if}\;\frac{N+1}{N}\leq a\leq\infty, \\
p*(a)=1, & r*(a)=\frac{a}{2-a} & {\rm if}\; 1\leq a\leq \frac{N+1}{N}.
\end{array}
\right.
\]
\end{defn}
The triplets of the form $(q,r,p)=(a,r^*(a),p^*(a))$ for 
$\frac{N+1}{N}\leq a<\infty$ are called endpoints.  
We note that the endpoint Strichartz estimate 
for the kinetic equation is false in all dimensions has been proved recently by 
Bennett, Bez, Guti\'{e}rrez and Lee~\cite{BBGL14}.

The solution of~\eqref{E:KT} can be written as 
\[
u=U(t)u_0+W(t)F
\]
where 
\[
U(t)u_0=u_0(x-vt,v)\;,\; W(t)F=\int_0^{t} U(t-s)F(s)ds.
\]
The estimates for the operator $U(t)$ and $W(t)$ respectively in the mixed Lebesgue norm 
$\|\cdot\|_{L^q_tL^r_xL^p_v}$ are called homogeneous and inhomogeneous Strichartz 
estimates. These two estimates together are given in the following Proposition where  
we use $p'$ to denote the conjugate exponent of $p$ and so on.
\begin{prop}[\cite{Ovc11},\cite{BBGL14}] 
Let $u$ satisfies~\eqref{E:KT}.  The estimate
\begin{equation}\label{E:Strichartz}
\|u\|_{L^q_tL^r_xL^p_v}\leq C(q,r,p,N)( \|u_0\|_{L^a_{x,v}} +\|F\|_{L^{\tilde{q}'}_tL^{\tilde{r}'}_xL^{\tilde{p}'}_v}  )
\end{equation}
holds for all $u_0\in L^a_{t,x}$ and all $F\in {L^{\tilde{q}'}_tL^{\tilde{r}'}_xL^{\tilde{p}'}_v} $ if and only if 
$(q,r,p)$ and $(\tilde{q},\tilde{r},\tilde{p})$ are two KT-admissible exponents triplets and $a=$HM$(p,r)=$HM$(\tilde{p}',\tilde{r}')$
with the exception of $(q,r,p)$ begin an endpoint triplet.
\end{prop}

We also need the estimates for the Boltzmann collision operator.
Recall that the collision operator can be split into gain and loss terms 
if the collision kernel satisfies Grad cut-off assumption~\eqref{D:Grad}. And it is convenient to introduce the bilinear 
gain term
\[
Q^+(f,g)(v)=\int_{v_*\in\mathbb{R}^n} f(v')g(v_*')B(v-v_*,\omega)
d\Omega(\omega) dv_*,
\] 
and the bilinear loss term
\[
Q^{-}(f,g)(v)=\int_{v_*\in\mathbb{R}^n} f(v)g(v_*)B(v-v_*,\omega).
d\Omega(\omega) dv_*
\]
The estimate we need for the gain term with the cut-off soft potential  is due 
to  Alonso, Carneiro and Gamba~\cite{ACG10}. 
\begin{prop}[\cite{ACG10}]\label{T:Convolution}
 Let $1<p_v,q_v,r_v <\infty$ with $-N<\gamma\leq 0$ and $1/p_v+1/q_v=1+\gamma/N+1/r_v$. Assume the kernel 
\[
B(v-v_*,\omega)=|v-v_*|^{\gamma}b(\cos\theta)
\]
with $b(\cos\theta)$ satisfies~\eqref{D:Grad}.
The bilinear operator $Q^{+}$ extends to a bounded operator from $L^{p_v}(\mathbb{R}^N)\times L^{q_v}(\mathbb{R}^N)\rightarrow 
L^{r_v}(\mathbb{R}^N)$ via the estimate 
\[
\|Q^{+}(f,g)\|_{L^{r_v}(\mathbb{R}^N)}\leq C\|f\|_{L^{p_v}(\mathbb{R}^N)}\|g\|_{L^{q_v}(\mathbb{R}^N)}.
\]
\end{prop}

The estimate for the  loss term in the  Lebesgue spaces is the following.
\begin{lem}\label{T:Loss}
Let $1<p_v,q_v,r_v <\infty$ with $-N<\gamma\leq 0$ and $1/p_v+1/q_v=1+\gamma/N+1/r_v$. Assume the kernel 
\[
B(v-v_*,\omega)=|v-v_*|^{\gamma}b(\cos\theta)
\]
with $b(\cos\theta)$ satisfies~\eqref{D:Grad}. The bilinear operator $Q^{-}$ is a bounded operator 
from $L^{p_v}(\mathbb{R}^N)\times L^{q_v}(\mathbb{R}^N)\rightarrow 
L^{r_v}(\mathbb{R}^N)$ via the estimate 
\[
\|Q^{-}(f,g)\|_{L^{r_v}(\mathbb{R}^N)}\leq C\|f\|_{L^{p_v}(\mathbb{R}^N)}\|g\|_{L^{q_v}(\mathbb{R}^N)}.
\]
\end{lem}
\begin{proof}
The case $\gamma=0$ is due to  H\"{o}lder inequality. For 
$-n<\gamma<0$, we note that for the cut-off case  $Q^{-}(f,g)=f(v)Lg(v)$ where 
\[ 
\begin{split}
Lg(v) &=\int_{v_*\in\mathbb{R}^N}\int_{S^{N-1}} |v-v_*|^{\gamma}b(\cos\theta) g(v_*)
d\Omega(\omega) dv_*\\
&=C \int_{v_*\in\mathbb{R}^N} |v-v_*|^{\gamma} g(v_*)dv_*
\end{split}
\]
is a convolution operator. Using H\"{o}lder inequality, we have 
\[
\|Q^{-}(f,g)(v)\|_{L^{r_v}(\mathbb{R}^n)}\leq \|f(v)\|_{L^{p_v}(\mathbb{R}^n)}
\|Lg(v)\|_{L^Z_v(\mathbb{R}^N)}\;,\; \frac{1}{r_v}=\frac{1}{p_v}+\frac{1}{Z}.
\]
Since $-N<\gamma<0$, we can invoke the Hardy-Littlewood-Sobolev inequality 
to have  
\[
\|Lg(v)\|_{L^Z(\mathbb{R}^N)}\leq C\|g\|_{L^{q_v}(\mathbb{R}^N)}
\]
where $-\frac{\gamma}{n}=1-(\frac{1}{q_v}-\frac{1}{Z})$ and end the proof. 
\end{proof}

Now we are ready to prove Theorem~\ref{T:1}, Corollary~\ref{Cor1} and Theorem~\ref{T:2}. 

\begin{proof} [Proof of Theorem~\ref{T:1}]
We define the solution map by
\begin{equation}\label{D:integral-equation}
\begin{split}
Sf(t,x,v)&=f_0(x-vt,v) 
+\int_0^t Q(f,f) (s,x-(t-s)v,v)  ds\\
&=U(t)f_0+\int_0^t U(t-s)Q(s)ds\\
&=U(t)f_0+W(t)Q(f,f)
\end{split}
\end{equation}
and wish to show that $S$ is a contraction mapping in the suitable Banach spaces. 
Applying the Strichartz estimates~\eqref{E:Strichartz} to above, we have 
\begin{equation}\label{E:solution-map}
\|Sf(t,x,v)\|_{L^{q}_tL^r_xL^{p}_v}\leq C{\big (}\|f_0\|_{L^a_{x,v}}+
\|Q(f,f)\|_{L^{\tilde{q}'}_tL^{\tilde{r}'}_xL^{\tilde{p}'}_v}{\big )}.
\end{equation}
The goal is to 
obtain the estimates of the from 
\begin{equation}\label{E:main-estimate}
\|Sf(t,x,v)\|_{X}\leq C_1\|f_0(x,v)\|_{Y}+ C_2 \|f(t,x,v)\|^2_X 
\end{equation}
where $X$ and $Y$ are suitable Banach spaces of the form 
$L^q_tL^r_xL^p_v$, $L^{a}_{x,v}$ respectively appearing in above estimates. 

By Proposition~\ref{T:Convolution},~Lemma~\ref{T:Loss} and~\eqref{E:main-estimate}, 
we wish to have 
\begin{equation}\label{Condition:11}
\frac{2}{p}=1+\frac{\gamma}{N}+\frac{1}{\tilde{p}'}
\end{equation}
for the estimate of $v$ variables. For $x$ variables, we need 
\begin{equation}\label{Condition:12}
2\tilde{r}'=r, \; r\geq 2
\end{equation}
for being able to apply the H\"{o}lder inequality.  Furthermore the Strichartz inequality demands the relation of pairs $(p,r),(\tilde{p}',\tilde{r}')$, 
\begin{equation}\label{Condition:13}
\frac{1}{p}+\frac{1}{r}=\frac{1}{\tilde{p}'}+\frac{1}{\tilde{r}'}.
\end{equation}
In order to apply the H\"{o}lder inequality to $t$ variable, we wish to have   
\begin{equation}\label{E:1t-variable}
\frac{2}{q}=\frac{1}{\tilde{q}'}<1,
\end{equation}
that is 
\begin{equation}\label{Condition:14}
\frac{2}{q}+\frac{1}{\tilde{q}}=1\;,\; \frac{1}{q}<\frac{1}{2}.
\end{equation}
Finally the KT-admissible conditions
\begin{align}
&\frac{1}{q}=\frac{N}{2}(\frac{1}{p}-\frac{1}{r})>0,\label{Condition:15}\\
&\frac{1}{\tilde{q}}=\frac{N}{2}(\frac{1}{\tilde{p}}-\frac{1}{\tilde{r}})>0
\label{Condition:16}
\end{align}
must be fulfilled.

We note that once $\gamma,p,r$ are given, $q,\tilde{p},\tilde{r},\tilde{q}$ 
are determined. Rewrite these conditions as      
\begin{subnumcases}{}
\frac{1}{p}+\frac{1}{r}=1+\frac{\gamma}{N} &\;\;{\rm from}\;~\eqref{Condition:11}
\;{\rm and\;}~\eqref{Condition:12},~\eqref{Condition:13}\label{Res:11} \\
\frac{1}{p}+\frac{1}{r}=\frac{2}{N} &\;{\rm from }\;~\eqref{Condition:14}\;
{\rm and}~\eqref{Condition:15},~\eqref{Condition:13} \label{Res:12}\\
0<\frac{1}{p}-\frac{1}{r}<\frac{1}{N} &\;{\rm from}\;$1/q<1/2$\;\;{\rm in}\;~\eqref{Condition:14}\;{\rm and}\;~\eqref{Condition:15}  \label{Res:13}\\
0<\frac{1}{p}-\frac{1}{r}< \frac{1}{2}(1+\frac{\gamma}{N})&
\;{\rm from}\;~\eqref{Condition:16}\;{\rm and}\;~\eqref{Condition:11} \label{Res:14}
\end{subnumcases}
Therefore
\[
\gamma=2-N\;,\;a=N.
\]
Thus we have 
\[
\frac{1}{N}<\frac{1}{p}<\frac{N+1}{N^2}\;,\;\frac{N-1}{N^2}<\frac{1}{r}<\frac{1}{N},
\]
and conclude the set   
\begin{equation}
\{ (p,r) |\;\frac{1}{N}<\frac{1}{p}<\frac{N+1}{N^2}
 \;,\;\frac{1}{r}=\frac{2}{N}-\frac{1}{p}\}.  
\end{equation}  

Using the triplets $(q,r,p),(\tilde{q}',\tilde{r}',\tilde{p}')$ satisfies above 
conditions, applying Proposition~\ref{T:Convolution} and Lemma~\ref{T:Loss} to  
the right hand side of~\eqref{E:solution-map} by choosing 
$r_v=\tilde{p}',\;p_v=q_v=p$ and the H\"{o}lder inequality to $x,t$ variables, 
we conclude that 
\begin{equation}\label{E:contraction1}
\|Sf\|_{L^q_tL^r_xL^p_v}\leq C_1\|f_0\|_{L^N_{x,v}}+C_2\|f\|^2_{L^q_tL^r_xL^p_v}
\end{equation}
with $(q,r,p)$ being determined as above.  With a similar argument one also obtains
\begin{equation}\label{E:uniqueness1}
\|Sf_1-Sf_2\|_{L^q_tL^r_xL^p_v}\leq C_2(\|f_1\|_{L^q_tL^r_xL^p_v}+\|f_2\|_{L^q_tL^r_xL^p_v})\|f_1-f_2\|_{L^q_tL^r_xL^p_v}.
\end{equation}

Let $\|f_0\|_{L^N_{x,v}}\leq {R}/{2C_1}$, $\overline{B}_{R}=\{f\in L^q_tL^r_xL^p_v |\; \|f\|_{L^q_tL^r_xL^p_v}\leq R \}$
where $R<1$ is small enough so that 
\begin{equation}\label{E:time1}
2C_2R<1,
\end{equation}
then from~\eqref{E:contraction1},~\eqref{E:uniqueness1} and~\eqref{E:time1} it follows that $L:\overline{B}_R\rightarrow \overline{B}_R$ is a 
contraction mapping and there exists a unique fixed point $f\in \overline{B}_R$ that is a solution to the integral equation~\eqref{D:integral-equation}.

Now we show that $f\in C([0,T],L^N_{x,v}),\;T\in[0,\infty]$. It has been noted by Ovcharov~\cite{Ovc09} that  
$U(t)f_0\in C(\mathbb{R};L^N_{x,v})$, hence it suffice to show that $W(t)$ is also continuous. 
Let $0\leq t\in (0,\infty]$.  Using inhomogeneous Strichartz with
$\tilde{q}',\tilde{r}',\tilde{p}'$ as above, we see that 
\[
\|W(t)Q(f,f)\|_{L^{\infty}([0,t];L^N_{x,v})}=\int_0^t \|U(t-s)Q(f,f)\|_{L^N_{x,v}}ds
\]
is bounded.  Since $U(t)$ is continuous, we conclude that $W(t)$ is continuous from above expression. The solution map $f_0\in B_{R}\subset L^N_{x,v}\rightarrow f\in L^q_tL^r_xL^p_v$ is Lipschitz continuous. For if $f$ and $g$ are two solutions 
with initial data $f_0$ and $g_0$ in $B_R$, we have as above 
\[
\|f-g\|_{L^q_tL^r_xL^p_v}\leq C_1\|f_0-g_0\|_{L^N_{x,v}}+C_2\|f-g\|^2_{L^q_tL^r_xL^p_v},
\]
and thus 
\[
\|f-g\|_{L^q_tL^r_xL^p_v}\leq C_3\|f_0-g_0\|_{L^N_{x,v}}.
\]

Next we show that the global solution $f$ scatters. We note that to show $\|f(t)-U(t)f_+\|_{L^N_{x,v}}\rightarrow 0$ as $t\rightarrow 
\infty$ is equivalent to
show that $\|U(-t)f(t)-f_+\|_{L^N_{x,v}}\rightarrow 0$ as $t\rightarrow 
\infty$  since $U(t)$ preserves the $L^N_{x,v}$ norm. 
By the Duhamel formula, we have 
\[
U(-t)f(t)=f_0+\int_0^t U(-s)Q(f,f)(s)ds.
\] 
Hence the scattering property is the consequence of the convergence of the integral 
\[
\int_0^{\infty} U(-t)Q(f,f)(t)dt
\]  
in $L^{N}_{x,v}$ in which case $f_+$ is given by
\begin{equation}\label{D:f-+}
f_+=f_0+\int_0^{\infty} U(-t)Q(f,f)(t)dt.
\end{equation}
 Let $U^*(t)$ be the adjoint operator of $U(t)$, it is clearly that $U^*(t)=U(-t)$. 
By duality, the homogeneous Strichartz estimate 
\[
\|U(t)g\|_{L^q_tL^r_xL^p_v}\leq C\|g\|_{L^{\frac{N}{N-1}}_{x,v}}
\]
with $(q,r,p)$ admissible and $1/p+1/r=(N-1)/N$ implies
\[
\|\int_0^{\infty} U^*(t)Q(f,f) \|_{L^{N}_{x,v}}\leq C\|Q(f,f)\|_{L^{q'}_tL^{r'}_xL^{p'}_v}
\]
with $1/p'+1/r'=N$. As the proof of existence of solution above, we see that 
the right hand side of above inequality is bounded by 
$\|f\|^2_{L^q_tL^r_xL^p_v}$ and thus bounded as $f\in \overline{B}_R$.
\end{proof}

\begin{proof}[Proof of Corollary~\ref{Cor1}]
First we exam the existence of $\Omega_+:f_+\rightarrow f_0$.
Using the second formula of Duhamel representation~\eqref{D:integral-equation} and 
the relation~\eqref{D:f-+}, we can write 
\begin{equation}\label{E:wave-operator}
f(t)=U(t)f_+-\int_t^{\infty}U(t-s)Q(f,f)(s)ds.
\end{equation}\label{D:Omega-+} 
Therefore the well-defined of $\Omega_+$ is equivalent to being able to define
~\eqref{E:wave-operator} for $t=0$ but this is just the reminiscence of global existence 
result above if $f_+$ is small enough in $L^N_{x,v}$.  

The map $\Omega_+$ is one to one
as a consequence of~\eqref{D:f-+} and uniqueness of solution $f$. 
This mapping is also surjective as a consequence of the fact the the 
small global solution scatters.
\end{proof}

\begin{proof}[Proof of Theorem~\ref{T:2}]
Let $\chi(r)$ be a smooth nonnegative bump even function supported on $-2\leq r\leq 2$
and satisfying $\chi(r)=1$ for $-1\leq r\leq 1$.  
Let $T>0$ be a positive number which will be chosen later. We define the solution map by
\begin{equation}\label{E:integral-equation-2}
\begin{split}
&Sf(t,x,v)\\
&=\chi(t/T)f_0(x-vt,v) 
+\chi(t/T)\int_0^t Q(f,f) (s,x-(t-s)v,v) d\tau ds\\
&=\chi(t/T)U(t)f_0+\chi(t/T)W(t)Q(f,f)
\end{split}
\end{equation}
and wish to show that $S$ is a contraction mapping in the suitable Banach spaces. 
Applying the Strichartz estimates~\eqref{E:Strichartz} to above, we have 
\[
\|Sf(t,x,v)\|_{L^{q}_tL^r_xL^{p}_v}\leq C{\big (}\|f_0\|_{L^a_{x,v}}+
\|Q(f,f)\|_{L^{\tilde{q}'}_tL^{\tilde{r}'}_xL^{\tilde{p}'}_v}{\big )}.
\]
The goal is to obtain the estimates of the from 
\[
\|Sf(t,x,v)\|_{X}\leq C_1\|f_0(x,v)\|_{Y}+ C_2 T^{\beta} \|f(t,x,v)\|_X^2 
\]
with $\beta>0$ where $X$ and $Y$ are suitable Banach spaces of the form 
$L^q_tL^r_xL^p_v$, $L^{a}_{x,v}$ respectively appearing in above estimates. 

The conditions posed on triplets $(q,r,p),(\tilde{q}',\tilde{r}',\tilde{p}')$
are similar. The only difference is the exponents about $t$ variables.  
For $t$ variable, the condition $\beta>0$ is equivalent to  
\begin{equation}\label{E:t-variable}
\frac{2}{q}<\frac{1}{\tilde{q}'}<1,
\end{equation}
that is 
\begin{equation}\label{Condition:4}
\frac{2}{q}+\frac{1}{\tilde{q}}<1\;,\; \frac{1}{q}<\frac{1}{2}.
\end{equation}

Therefore we conclude a system of restrictions similar to that of Theorem~\ref{T:1}.    
\begin{subnumcases}{}
\frac{1}{p}+\frac{1}{r}=1+\frac{\gamma}{N} \label{Res:1} \\
\frac{1}{p}+\frac{1}{r}<\frac{2}{N}  \label{Res:2}\\
0<\frac{1}{p}-\frac{1}{r}<\frac{1}{N}   \label{Res:3}\\
0<\frac{1}{p}-\frac{1}{r}< \frac{1}{2}(1+\frac{\gamma}{N}). \label{Res:4}
\end{subnumcases}
The conditions~\eqref{Res:1} and~\eqref{Res:2} imply that $-N<\gamma<-(N-2)$. 
 Thus~\eqref{Res:3} holds by~\eqref{Res:4}. Since 
\[
\frac{1}{a}=\frac{1}{2}\cdot\frac{\gamma+N}{N}<\frac{1}{N}<\frac{N}{N+1},
\] 
we also require
\[
\begin{split}
& \frac{1}{2}\cdot\frac{\gamma+N}{N}\leq \frac{1}{p} < \frac{N+1}{N}\cdot \frac{1}{2}\cdot\frac{\gamma+N}{N} \\
& \frac{N-1}{N}\cdot\frac{1}{2}\cdot\frac{\gamma+N}{N}< \frac{1}{r} \leq  \frac{1}{2}\cdot\frac{\gamma+N}{N}
\end{split}
\] 
for the KT-admissible condition. 
Thus we conclude the set
 \begin{equation}\label{E:pr-set}
  \{(p,r)| \; \frac{1}{p}=\alpha\frac{(\gamma+N)}{N},\;\frac{1}{r}=(1-\alpha)\frac{(\gamma+N)}{N}\;,\;{\rm with}\;\frac{1}{2}<\alpha<\frac{N+1}{2N}  \}
 \end{equation}
satisfies all the conditions list above. Ans it is easy to check that 
$(q,r,p)$ and $(\tilde{q},\tilde{r},\tilde{p})$ are KT-admissible triplets when 
$(p,r)$ lies in set~\eqref{E:pr-set}.

Using triplets $(q,r,p),(\tilde{q}',\tilde{r}',\tilde{p}')$ above 
and the argument as the Theorem~\ref{T:1}, we conclude  
\begin{equation}\label{E:contraction}
\begin{split}
& \|Sf(t,x,v)\|_{L^q([0,T];L^r_xL^p_v)}\\
& {\hskip 1cm}\leq C_1\|f_0(x,v)\|_{L^a_{x,v}}+ C_2 T^{\beta} 
\|f(t,x,v)\|^2_{L^q([0,T];L^r_xL^p_v)}
\end{split}
\end{equation}
where
\[
\beta=\frac{(2-N)-\gamma}{2}>0.
\] 
 With a similar argument one also obtains
\begin{equation}\label{E:uniqueness}
\begin{split}
& \|Lf_1-Lf_2\|_{L^q([0,T];L^r_xL^p_v)}\\
&{\hskip .5cm}\leq C_2T^{\beta}
(\|f_1\|_{L^q([0,T];L^r_xL^p_v)}+\|f_2\|_{L^q([0,T];L^r_xL^p_v)})\|f_1-f_2\|_{L^q([0,T];L^r_xL^p_v)}.
\end{split}
\end{equation}
Let $R=2C_1\|f_0\|_{L^a_{x,v}}$ be any positive number, $B_{R}=\{f|\; \|f\|_{L^q_tL^r_xL^p_v}\leq R \}$ and $T$ such that 
\begin{equation}\label{E:time}
C_2T^{\beta}R<\frac{1}{2},
\end{equation}
then from~\eqref{E:contraction},~\eqref{E:uniqueness} and~\eqref{E:time} it follows that $L:B_R\rightarrow B_R$ is a 
contraction mapping and there exists a unique fixed point $f\in B_R$ that is a solution to the integral 
equation~\eqref{E:integral-equation-2}.

Finally,  we show that the uniqueness of solution.  If $f_1$ and $f_2$ are two solutions, 
is easy to see that we have
\[
\begin{split}
&\|f_1-f_2\|_{L^q([0,t];L^r_xL^p_v)}\\
&\leq C_2t^{\beta}(\|f_1\|_{L^q([0,T];L^r_xL^p_v)}+\|f_2\|_{L^q([0,T];L^r_xL^p_v)})\|f_1-f_2\|_{L^q([0,t];L^r_xL^p_v)}.
\end{split}
\] 
for $0<t\leq T$.  By choosing $t$ small enough, we have 
\[
\|f_1-f_2\|_{L^q([0,t];L^r_xL^p_v)}\leq \frac{1}{2}\|f_1-f_2\|_{L^q([0,t];L^r_xL^p_v)}
\]
and thus $f_1=f_2$ on $[0,t]$. We can cover the interval $[0,T]$ by iterates this argument. 
The well-posedness of  this case ends. 

\end{proof}

\noindent{\bf Acknowledgments.}
J.-C. Jiang was supported in part by 
 National Sci-Tech Grant MOST 105-2115-M-007-005,
Mathematics Research Promotion Center and  
National Center for Theoretical Sciences. The authors would like to thank anonymous referees 
for helpful comments.

\end{document}